\documentclass[11pt,oneside]{amsart}

\title[Conjugacy and homogeneous digraphs]{The conjugacy problem for automorphism groups of homogeneous digraphs}

\author{Samuel Coskey} \address{Samuel Coskey, Department of Mathematics, Boise State University, 1910 University Drive, Boise, ID, 83725}
\email{scoskey@nylogic.org}
\urladdr{boolesrings.org/scoskey}

\author{Paul Ellis} \address{Paul Ellis, Department of Mathematics and Computer Science, Manhattanville College, 2900 Purchase Street, Purchase, NY, 10577}
\email{paulellis@paulellis.org}
\urladdr{paullellis.org}

\usepackage[marginratio=1:1]{geometry}
\usepackage{setspace}
\usepackage{mathpazo,amssymb,bm}
\usepackage{tikz}
\usepackage{hyperref}

\newcommand{\NN}{\mathbb N}
\newcommand{\ZZ}{\mathbb Z}
\newcommand{\QQ}{\mathbb Q}
\newcommand{\set}[1]{\left\{\,#1\,\right\}}
\DeclareMathOperator{\Aut}{Aut}
\newcommand{\semistar}{\mathbin{\hat{*}}}

\usepackage{etoolbox}
\makeatletter\pretocmd{\@seccntformat}{\S}{}{}
  \pretocmd{\@subseccntformat}{\S}{}{}\makeatother

\newtheorem{thm}{Theorem}[section]
\newtheorem{lem}[thm]{Lemma}
\newtheorem{prop}[thm]{Proposition}

\theoremstyle{definition}
\newtheorem{defn}[thm]{Definition}

\begin{document}
\maketitle

\begin{abstract}
  We decide the Borel complexity of the conjugacy problem for automorphism groups of countable homogeneous digraphs. Many of the homogeneous digraphs, as well as several other homogeneous structures, have already been addressed in \cite{summer} and \cite{conjugacy1}. In this article we complete the program, and establish a dichotomy theorem that this complexity is either the minimum or the maximum among relations which are classifiable by countable structures. We also discuss the possibility of extending our results beyond graphs to more general classes of countable homogeneous structures.
\end{abstract}

\section{Introduction}

This article is a contribution to the study of the automorphism groups of countable homogeneous digraphs. We use the term \emph{digraph} in the model-theoretic sense to mean an oriented simple graph. A countable digraph is said to be \emph{homogeneous} if every finite partial automorphism extends to a total automorphism. A survey of the study of countable homogeneous structures can be found in \cite{macpherson}; the countable homogeneous digraphs are classified in \cite{cherlin}.

Our main result will be stated in terms of the Borel complexity theory of equivalence relations. We recall that if $E,F$ are equivalence relations on standard Borel spaces $X,Y$ then we say that $E$ is  \emph{Borel reducible} to $F$ if there is a Borel function $f\colon X\to Y$ such that $x\mathrel{E}x'\iff f(x)\mathrel{F}f(x')$. An equivalence relation $E$ is said to be \emph{smooth} if it is Borel reducible to the equality relation on $\mathbb R$. An equivalence relation $E$ is said to be \emph{Borel complete} if it is Borel reducible to an isomorphism relation on a class of countable structures, and conversely any isomorphism relation on a class of countable structures is Borel reducible to $E$. (Here the countable structures are coded by a sequence of relations on $\NN$.) We will use the standard fact that the isomorphism relations on the classes linear orders and partial orders are Borel complete \cite{FS}. For a resource on Borel complexity theory we refer the reader to \cite{gao}.

In \cite{summer} and \cite{conjugacy1}, we sought to compute the Borel complexity of the \emph{conjugacy relation} on automorphism groups of numerous countable homogeneous structures. For countable homogeneous digraphs, we were able to decide this complexity in all but three cases, which turned out to be more difficult than the rest. For every digraph that we did analyze, the complexity turned out to be either smooth or Borel complete. In this article we show that in the three remaining cases the conjugacy problems are all Borel complete as well. This completes the proof of the dichotomy:

\begin{thm}
  \label{thm:main}
  If $G$ is a countable homogeneous digraph then the conjugacy problem for $\Aut(G)$ is either smooth or Borel complete.
\end{thm}

As we have said, to complete the proof it remains to consider just the three remaining digraphs. These are the generic partial order $\mathcal P$ (Section~2), the generic shuffled partial order $\mathcal P(3)$ (Section~3), and the semigeneric complete multipartite digraph $\infty\semistar I_\infty$ (Section~4). These three proofs showcase many of the tools used in \cite{conjugacy1} together with some additional tricks.

The conclusion of the proof of Theorem~\ref{thm:main} suggests one should next ask about the classification of automorphisms of other countable homogeneous structures. In Section~5, we introduce the class of homogeneous structures with the ($n$-ary) \emph{Borel amalgamation property}, and show how to generalize our methods to apply to such structures.

\textbf{Acknowledgement}. We would like to thank Gregory Cherlin for helpful discussions regarding this material.

\section{The generic partial order}

Let $\mathcal P$ denote the generic countable homogeneous partial order. We refer the reader to \cite{schmerl} for the classification of all homogeneous partial orders. In this section we show that the conjugacy problem for $\Aut(\mathcal P)$ is Borel complete. We begin with the following lemma giving a strong form of the amalgamation property for the class of partial orders.

\begin{lem}
  \label{lem:po-amalg}
  Let $P$ be a partial order and for each $i$ let $P\cup\{a_i\}$ be a partial order extending $P$. Then the transitive closure $P'$ of $P\cup\{a_0,a_1,\ldots\}$ is again a partial order.  Also, $P'$ adds no new relations to $P$.
\end{lem}

\begin{proof}
  First note that the transitive closure is obtained in just one step by adding a relation $a_i\leq a_j$ whenever there is $p\in P$ such that $a_i\leq p\leq a_j$. To verify transitivity holds after performing this step, suppose that $a_i\leq a_j\leq a_k$. Then there exist $p,q\in P$ such that $a_i\leq p\leq a_j\leq q\leq a_k$. Since $P\cup\{a_j\}$ is a partial order, we must have that $p\leq q$. Therefore $a_i\leq p\leq a_k$ and it follows that $a_i\leq a_k$.

  Now suppose towards a contradiction that antisymmetry fails in $P'$. Since we have added no new relations within $P$ or between elements of $P$ and elements of $\{a_i\}$ there must be distinct $i,j$ such that $a_i\leq a_j\leq a_i$. Then there are $p,q\in P$ such that $a_i\leq p\leq a_j\leq q\leq a_i$. It follows that $p\leq q\leq p$, so by antisymmetry in $P$, we have $p=q$. But now $a_i\leq p\leq a_i$, which contradicts antisymmetry in $P\cup\{a_i\}$.
\end{proof}

\begin{thm}
  \label{thm:partial}
  The class of countable partial orders is Borel reducible to the conjugacy relation on $\Aut(\mathcal P)$. Hence the conjugacy relation on $\Aut(\mathcal P)$ is Borel complete.
\end{thm}

\begin{proof}
  Given a countable partial order $P$, we build a copy $Q_P$ of $\mathcal P$ and an automorphism $\phi_P$ of $Q_P$ in such a way that $P\cong P'$ if and only if $\phi_P$ is conjugate to $\phi_{P'}$. This will be done in stages $Q_P^n$, $\phi_P^n$. To begin, let $Q_P^0$ consist of $\ZZ$ many incomparable copies of $P$ together with $\NN$ many copies of $\ZZ$. Denote these latter copies $\ZZ^{(i)}=\{m^{(i)}\mid m\in\ZZ\}$ for each $i\in\NN$. Let $\phi_P^0$ act on $Q_P^0$ by sending the $i$th copy of $P$ to the $(i+1)$st copy of $P$, and sending $m^{(i)}$ to $(m+1)^{(i)}$.

  Assume that $Q_P^n$ and $\phi_P^n$ have been constructed, and construct $Q_P^{n+1}$ and $\phi_P^{n+1}$ as follows. Begin by considering each set of constraints of the form $\bar a<x<\bar b$ and $x\perp \bar c$ which are consistent with the axioms of a partial order and the diagram of $Q_P^n$. (For example, we assume that $\bar a<\bar b$, $\bar c\not\leq\bar a$, and so forth.) Additionally assume
  \begin{itemize}
  \item[($\star$)] the constraints contain relations the form $m^{(i)}<x<(m+1)^{(i)}$.
  \end{itemize}
  For each such set of constraints, we add a new realization $x$ to $Q_P^{n+1}$. We do not add any relations involving $x$ except those implied by the constraints and transitivity. Then, we close $Q_P^{n+1}$ under transitivity. By Lemma~\ref{lem:po-amalg}, we have added no new edges to $Q_P^n$.

  We also extend $\phi_P^n$ to an automorphism $\phi_P^{n+1}$ of $Q_P^n$ in the obvious way: If $x$ is the point corresponding to the parameters $\bar a,\bar b,\bar c$, then we let $\phi_P^{n+1}(x)$ be the point corresponding to the parameters $\phi_P^n(\bar a),\phi_P^n(\bar b),\phi_P^n(\bar c)$.

  We claim that any element of $Q_P$ is related to just finitely many copies of $\ZZ$ in $Q_P^0$. Clearly this claim holds for elements of $Q_P^0$ itself. Next assume that the claim holds for elements of $Q_P^n$ and consider constraints of the form $\bar a<x<\bar b$ and $x\perp\bar c$ with parameters from $Q_P^n$. By hypothesis the elements $\bar a,\bar b,\bar c$ are related to just finitely many copies of $\ZZ$ in $Q_P^0$ among them all. Adding an element $x$ satisfying this constraint (as done above) and closing under transitivity does not result in $x$ being related to any additional copies of $\ZZ$. This completes the proof of the claim.

  Now to see that $Q_P$ satisfies the one-point extension property, let $\bar a<x<\bar b$ and $x\perp\bar c$ be an arbitrary set of constraints consistent with the axioms of a partial order. Let $Q_P^n$ be the least level containing all of the parameters $\bar a,\bar b,\bar c$. By the claim, these parameters are related to just finitely many copies of $\ZZ$ in $Q_P^0$ among them all. By the argument of the claim, it is possible to add a realization $x$ to $Q_P^n$ which is related just to these finitely many copies of $\ZZ$ in $Q_P^0$. Let $\ZZ^{(i)}$ be the first copy of $\ZZ$ in $Q_P^0$ not related to $x$. We now consider the constraints $\bar a<y<\bar b$ and $y\perp\bar c$ and $0^{(i)}<y<1^{(i)}$. This extended set of constraints is consistent and of the form $(\star)$. Hence in the construction we have placed a realization $y$ of the extended constraints into $Q_P^{n+1}$. This completes the verification that $Q_P$ satisfies the one-point extension property.

  Towards a conclusion, observe that for every $x$ in a copy of $P$ in $Q_P^0$ we have that the $\phi_P$-orbit of $x$ is an antichain. On the other hand for every other element $x$ we have that the $\phi_P$-orbit of $x$ is a chain. This is because we have some constraint of the form $0^{(i)}<x<1^{(i)}$. This implies $1^{(i)}=\phi(0^{(i)})<\phi(x)$, and the two together imply that $x<\phi(x)$.

  Thus we can recover the copies of $P$ in $Q_P^0$ as the set of points whose $\phi_P$-orbit is an antichain, and then further recover $P$.
  Hence if $\alpha\colon Q_P \cong Q_{P'}$ and $\alpha\phi_P=\phi_{P'}\alpha$ it follows that $\alpha$ restricts to an isomorphism $Q_P^0 \cong Q_{P'}^0$ that sends $\phi_P$-orbits to $\phi_{P'}$-orbits. Therefore by passing to the quotient graphs of $Q_P^0,Q_{P'}^0$ by the $\phi_P$ and $\phi_{P'}$-orbit equivalence relations, we see that $\alpha$ induces an isomorphism  $P\cong P'$.

  To conclude, we remark briefly on how the construction can be exhibited in a Borel fashion. We fix the underlying sets of $P,Q_P,\mathcal{P}$ to be $\NN$. The construction of $Q_P$ can be made Borel by reserving an infinite subset $I_n\subset\NN$ for each $Q_P^n$, and using a previously fixed enumeration of the finite subsets $S\subset I_k$. This immediately implies that the construction of $\phi_P$ is Borel also. Finally we can regard $\phi_P$ as an automorphism of $\mathcal{P}$ using a back-and-forth construction between $Q_P$ and $\mathcal{P}$, where each choice in the construction is resolved by choosing the least available witness.
\end{proof}

\section{The generic shuffled partial order}

The generic shuffled partial order, denoted $\mathcal P(3)$, is a graph obtained by ``shuffling'' three disjoint dense subsets of $\mathcal P$ in the following fashion.

\begin{defn}
  \label{defn:p3}
  Let $\mathcal P=P_0\sqcup P_1\sqcup P_2$ be a partition of $\mathcal P$ into three dense subsets. Define the shuffled graph $\mathcal P(3)$ on the underlying set of $\mathcal P$ as follows. First, if $x,y\in P_i$, then set $x\to_{\mathcal P(3)}y\iff x<_{\mathcal P}y$. Next for each $i\in\ZZ/3\ZZ$, if $x\in P_i$ and $y\in P_{i+1}$ then set
\begin{align*}
x\to_{\mathcal P(3)} y &\iff x>_{\mathcal P} y\\
x\leftarrow_{\mathcal P(3)} y &\iff x\perp_{\mathcal P} y\\
x\perp_{\mathcal P(3)} y &\iff x<_{\mathcal P} y
\end{align*}
\end{defn}

See Chapter~5 of \cite{cherlin} for the proof that $\mathcal P(3)$ is homogeneous. We remark that the construction of $\mathcal P(3)$ is similar to that of the digraph $S(3)$, which is obtained by shuffling three disjoint dense subsets of $\QQ$. See Section~2.2 of \cite{conjugacy1} for our treatment of $S(3)$.

The argument of the previous section can be modified to show that the conjugacy problem for $\Aut(\mathcal P(3))$ is again Borel complete. In the proof we will let $\mathcal P_3$ denote the generic three-colored partial order. (Finite three-colored partial orders form an amalgamation class.) The structure $\mathcal P_3$ can be viewed simply as a copy of $\mathcal P$ partitioned into three distinguished dense subsets $P_0,P_1,P_2$.

\begin{thm}
  The isomorphism relation on countable partial orders is Borel reducible to the conjugacy relation on $\Aut(\mathcal P(3))$. Hence the conjugacy relation on $\Aut(\mathcal P(3))$ is Borel complete.
\end{thm}

\begin{proof}
  Given a partial order $P$ we modify the proof of Theorem~\ref{thm:partial} to build a copy $Q_P$ of $\mathcal P_3$ as follows. The first level $Q_P^0$ is constructed as before, with all vertices of $Q_P^0$ of color $0$. When constructing $Q_P^{n+1}$ we again add elements $x$ satisfying each admissible constraint; in this way elements $x$ of all three colors will be added. Continuing the construction as before, we obtain a structure $Q_P$ which is a copy of $\mathcal P_3$, and an automorphism $\phi_P$ of $Q_P$. It is again the case that $x\in Q_P^0$ if and only if the $\phi_P$-orbit of $x$ is an antichain.

  The structure $Q_P$ gives rise to a corresponding copy $Q_P(3)$ of $\mathcal P(3)$ obtained by shuffling the colors according to the rules in Definition~\ref{defn:p3}. Since $\phi_P$ preserves the colors of $Q_P$, it is easy to see that $\phi_P$ is an automorphism of $Q_P(3)$ too.

  For the forward implication of the Borel reduction, beginning with an isomorphism $\alpha\colon P\cong P'$ it gives rise to an automorphism of the first level $Q_P^0$ (which we recall was monochromatic in color $0$). This then extends naturally to a color-preserving automorphism of all $Q_P$, and hence to an automorphism of $Q_P(3)$. The resulting extension $\bar\alpha$ conjugates $\phi_P$ to $\phi_{P'}$.

  For the converse implication, we note that since all the $\phi_P$-orbits are monochromatic, they retain their structure even after the shuffle. That is, antichain orbits remain antichain orbits, and non-antichain orbits remain non-antichain orbits. Thus given $\phi_P$ we can recover a copy of $Q_P^0$ as the set of antichain orbits and conclude as in the proof of Theorem~\ref{thm:partial}.
\end{proof}

\section{The semigeneric complete multipartite digraph}

We say a digraph is \emph{complete multipartite} if its $\perp$ relation is an equivalence relation. The class of finite such graphs forms an amalgamation class, and we write $\infty*I_\infty$ for the countable generic complete multipartite graph. In this section we consider the following ``semigeneric'' variant of $\infty*I_\infty$. Let $\mathcal C$ be the class of finite complete multipartite graphs which additionally satisfy the following \emph{parity property}: for every two maximal antichains $A$ and $B$ and every distinct $a,a'\in A$ and distinct $b,b'\in B$ there is an even number of edges pointing from the set $\{a,a'\}$ to the set $\{b,b'\}$. This is once again an amalgamation class, as shown in \cite{cherlin-imprimitive}. We let $\infty\semistar I_\infty$ denote the countable generic graph corresponding to the class $\mathcal C$.

In \cite{conjugacy1} we showed that the conjugacy problem for $\Aut(\infty*I_\infty)$ is Borel complete. In this section we show that the conjugacy problem for $\Aut(\infty\semistar I_\infty)$ is again Borel complete. We remark that the argument for completeness of $\Aut(\infty*I_\infty)$ cannot be used directly for $\Aut(\infty\semistar I_\infty)$, since the widget used in the proof did not have the parity property (see Figure~2 of \cite{conjugacy1}). The following lemma will help us understand the structure $\infty\semistar I_\infty$ and its automorphisms.

\begin{lem}
  \label{lem:semigeneric-structure}
  Let $G$ be a complete multipartite digraph with the parity property, and suppose $A, B$ are two maximal antichains. There exists subsets $R_{AB}\subset A$ and $S_{AB}\subset B$ such that for $x\in A$ and $y\in B$, we have $x\to y$ if and only if we have that $x\in R_{AB}\iff y\in S_{AB}$. Refer to Figure~\ref{fig:semigen} for a diagram of these relationships.
\end{lem}

\begin{figure}[h]
  \begin{tikzpicture}
    \node (a) at (0,0) {$R_{AB}$};
    \node (b) at (2,0) {$S_{AB}$};
    \node (c) at (0,-1) {$R_{AB}^c$};
    \node (d) at (2,-1) {$S_{AB}^c$};
    \draw[->,double] (a) -- (b);
    \draw[->,double] (b) -- (c);
    \draw[->,double] (c) -- (d);
    \draw[->,double] (d) -- (a);
  \end{tikzpicture}
  \caption{The edge relationships between the sets $R_{AB}$, $S_{AB}$, and their complements $R_{AB}^c=A\smallsetminus R_{AB}$ and $S_{AB}^c=B\smallsetminus S_{AB}$.\label{fig:semigen}}
\end{figure}
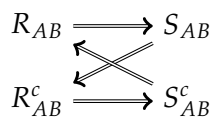

\begin{proof}
  Let $A, B$ be two maximal antichains and, if it is possible, fix $x_0\in A$ and $y_0\in B$ such that $x_0\to y_0$. We then let $R_{AB}=\set{x\in A\mid x\to y_0}$ and let $S_{AB}=\set{y\in B\mid x_0\to y}$. We claim these sets have the desired properties. Indeed, applying the parity property to the pairs $\{x_0,x\}$ and $\{y_0,y\}$ one can conclude that $x\to y$ in the cases when $x\in R_{AB}$ and $y\in S_{AB}$ or else $x\notin R_{AB}$ and $y\notin S_{AB}$. Similarly one can conclude that $x\leftarrow y$ in the cases when $x\in R_{AB}$ and $y\notin S_{AB}$ or else $x\notin R_{AB}$ and $y\in S_{AB}$, as desired.

  If it is not possible to pick $x_0,y_0$ above, then instead pick $x_0\in A$ and $y_0\in B$ with $x_0\leftarrow y_0$ and proceed similarly to find $R_{BA}$ and $S_{BA}$. Then simply let $R_{AB}:=S_{BA}$ and $S_{AB}:=R_{BA}^c$.
\end{proof}

\begin{thm}
  The isomorphism relation for countable linear orders is Borel reducible to the conjugacy problem for $\Aut(\infty\semistar I_\infty)$. Hence the conjugacy problem for $\Aut(\infty\semistar I_\infty)$ is Borel complete.
\end{thm}

\begin{proof}
  Given a countable linear order $L$ we will recursively construct a copy $G_L$ of $\infty\semistar I_\infty$ together with an automorphism $\phi_L$ of $G_L$ such that $L\cong L'$ if and only if $\phi_L$ is conjugate to $\phi_{L'}$. To begin we let $G_L^0=L$ itself, and $\phi_L^0=$ the identity on $G_L^0$. We remark that $G_L^0$ is multipartite with parts of size $1$, and therefore it has the parity property. In the construction we will also require a linear ordering $<_L^n$ on $G_L^n$, and we initially set $<_L^0$ equal to the given ordering of $L$.

  For the remainder of the proof we fix an enumeration $\tau_k(x,\bar{y})$ of the types of the theory of digraphs.

  Assuming $G_L^n$, $\phi_L^n$, $<_L^n$ have been constructed, we construct $G_L^{n+1}$, $\phi_L^{n+1}$, $<_L^{n+1}$ as follows. We consider in turn each pair $k\in\mathbb{N}$ and $S\in(G_L^n)^{<\infty}$ such that the parameterized type $\tau_k(x,S)$ does not contradict the parity property. For each such pair, we put three new points $a_k(S),b_k(S),c_k(S)$ into $G_L^{n+1}$ satisfying $\tau_k(x,S)$.

  If $\tau_k(x,S)$ forces $x$ to be in a maximal antichain $A$ of $G_L^n$, then we will place these new points into $A$. Otherwise we will create three new antichains with one point each. Formally, if $\tau_k(x,S)$ contains a formula of the form $x\perp s$, we set $a_k(S)\perp b_k(S)\perp c_k(S)\perp a_k(S)$. For future reference, let us say say that the type $\tau_k(x,S)$ and the points $a_k(S),b_k(S),c_k(S)$ added in this manner are of \emph{Class~1}. On the other hand, if $\tau_k(x,S)$ does not contain a formula of the form $x\perp s$, then we set $a_k(S)\to b_k(S)\to c_k(S)\to a_k(S)$. We say that the type $\tau_k(x,S)$ and the points $a_k(S),b_k(S),c_k(S)$ added in this manner are of \emph{Class~2}.

  Before we describe the rest of the edges of $G_L^{n+1}$, let us use Lemma~\ref{lem:semigeneric-structure} to define the sets $R_{AB}$ and $S_{AB}$ for every pair of maximal antichains $A,B$ of $G_L^n$. Note that there is an ambiguity in the lemma, since $R_{AB}$ and $S_{AB}$ may be swapped with their complements. To make the definition uniform throughout our construction, note that each maximal antichain has an element $e_A$ which was added \emph{earliest}. We always choose $R_{AB}$ so that it contains this element $e_A$.

  Now given any element $a$ of Class~1, $a$ was added to some maximal antichain $A$. For each other maximal antichain $B$, we additionally specify that $a$ lies in $R_{AB}$ unless its type $\tau_K(S)$ explicitly forces us to put $a$ into $R_{AB}^c$. This specification determines the remaining edges between elements of Class~1 and elements of $G_L^n$, and also between any two elements of Class~1.

  Next given an element $a$ of Class~2, $a$ lies in its own maximal antichain $A=\{a\}$. For each other element $b$ of $G_L^n$ or of Class~1 we set $a\to b$ unless the type of $a$ explicitly forces us to set $a\leftarrow b$.

  It remains only to define the edges between two elements of Class~2. For this we will need to define the linear ordering $<_{n+1}$. First let $\prec_n$ be the lexicographic ordering of $(G_L^n)^{<\infty}$ inherited from $<_n$. We then make the definitions:
  \begin{itemize}
  \item If $d\in G_L^n$ and $d'\in G_L^{n+1}\smallsetminus G_L^n$, set $d<_{n+1}d'$
  \item If $k<k'$, set $a_k(S),b_k(S),c_k(S)<_{n+1}a_{k'}(S), b_{k'}(S)c_{k'}(S)$
  \item If $k\in\mathbb{N}$ and $S\prec_n S'$, set $a_k(S),b_k(S),c_k(S)<_{n+1} a_k(S'), b_k(S'),c_k(S')$
  \item If $k\in\mathbb{N}$ and $S\in(G_L^n)^{<\infty}$, set $a_k(S)<_{n+1} b_k(S)<_{n+1} c_k(S)$
  \end{itemize}
  Now if $\tau_k(S)$ and $\tau_{k'}(S')$ are two types of Class~2, we set $a_k(S),b_k(S),c_k(S)\rightarrow a_{k'}(S'),b_{k'}(S'),c_{k'}(S')$ precisely when $a_k(S)<_{n+1}a_{k'}(S')$.

  Finally we extend the automorphism $\phi_L^n$ of $G_L^n$ to an automorphism of $G_L^{n+1}$. Given a type $\tau_k(x,S)$ of the form considered above, we let $S'=\phi_n(S)$. Then we let $\phi_L^{n+1}$ map $a_k(S)\mapsto b_k(S')$, $b_k(S)\mapsto c_k(S')$, and $c_k(S)\mapsto a_k(S')$. This completes the construction.

  Letting $G_L=\bigcup G_L^n$ we have that $G_L$ is complete multipartite, satisfies the parity property, and has the one-point extension property with respect to this class. Thus $G_L$ is a copy of $\infty\semistar I_\infty$. Letting $\phi_L=\bigcup\phi_L^n$ we have that $\phi_L$ is an automorphism of $G_L$. As in our previous arguments, it is not hard to see that an isomorphism $L\cong L'$ gives rise to a conjugacy between $\phi_L$ and $\phi_{L'}$.  Moreover we can recover $L$ as the set of fixed points of $\phi_L$, which guarantees that a conjugacy between $\phi_L$ and $\phi_{L'}$ gives rise to an isomorphism $L\cong L'$.
\end{proof}

We conclude this section with a note motivating the need for the involved proof of the previous theorem. The ``bowtie'' structure of Figure~\ref{fig:semigen} has an automorphism swapping $R_{AB}$ with $R_{AB}^c$ and $S_{AB}$ with $S_{AB}^c$. One might expect that this symmetry can be used to build automorphisms of $\infty\semistar I_\infty$. However the following result shows that these partial automorphisms do not extend to $\infty\semistar I_\infty$, necessitating the more complicated construction above.

\begin{prop}
  If $\phi$ is an automorphism of $\infty\semistar I_\infty$ which fixes each maximal antichain setwise then $\phi$ is the identity.
\end{prop}

\begin{proof}
  Suppose $\phi$ fixes each antichain setwise, and assume towards a contradiction that $a,\phi(a)$ are unequal and lie in the same maximal antichain $A$.

  Case 1: $\phi^2(a)\neq a$. Then by genericity there exists a maximal antichain $B$ such that $a,\phi(a)\in R_{AB}$ and $\phi^2(a)\in R_{AB}^c$. Then $a,\phi(a)\in R_{AB}$ implies that $\phi(R_{AB})=R_{AB}$. But on the other hand $\phi(a)\in R_{AB},\phi^2(a)\in R_{AB}^c$ implies that $\phi(R_{AB})=R_{AB}^c$. This is a contradiction.

  Case 2: $\phi^2(a)=a$ and there exists another element $a'\in A$ besides $a,\phi(a)$ such that $\phi(a')\neq a'$. Then there exists a maximal antichain $B$ such that $a,\phi(a),a'$ lie in $R_{AB}$ and $\phi(a')$ lies in $R_{AB}^c$. Then we reach a contradiction similarly to Case~1.

  Case 3: $\phi^2(a)=a$ and some other element $a'\in A$ besides $a,\phi(a)$ is fixed by $\phi$. Then there exists a maximal antichain $B$ such that $a$ lies in $R_{AB}$ and $\phi(a),a'$ lie in $R_{AB}^c$. We again reach a contradiction similarly to Case~1.
\end{proof}

\section{Toward a more general theorem}


In this final section we discuss the possibility of generalizing the methods of \cite{summer,conjugacy1} and the present article to classes of homogeneous structures other than graphs. More specifically, observe that many of our results are established by finding a Borel reduction from the isomorphism relation on the class of countable substructures of some countable homogeneous structure $M$ to the conjugacy relation on $\Aut(M)$. It is natural to ask whether it is always possible to find such a reduction.

Recently, Bilge and Melleray showed that the answer is yes in the case when $M$ has a property which is stronger than homogeneity. In order to describe this result, let us recall some notions from Fra\"iss\'e theory. Let $\mathcal K$ be a class of finite (relational) structures, and let $\mathcal K_\omega$ be the class of countable structures, all of whose finite substructures lie in $\mathcal K$. We say that $\mathcal K$ has the \emph{amalgamation property} (AP) if for any $A,B_1,B_2\in\mathcal K$ and embeddings $g_i\colon A\to B_i$, there exists $C\in\mathcal{K}$ and embeddings $f_i\colon B_i\to C$ satisfying 
\[f_1 \circ g_1 = f_2 \circ g_2\text{.}
\]
We recall that $\mathcal K$ has the amalgamation property if and only if there exists a (necessarily unique) homogeneous structure $M\in\mathcal K_\omega$, called the \emph{Fra\"iss\'e limit} of $\mathcal K$, such that the class of finite substructures of $M$ is exactly $\mathcal K$.

Next we say that $\mathcal{K}$ has the \emph{strong amalgamation property} (SAP) if in the definition of AP we additionally have $f_1(B_1)\cap f_2(B_2)=f_1\circ g_1(A)=f_2\circ g_2(A)$. Finally we say that $\mathcal{K}$ has the \emph{free amalgamation property} (FAP) if in the definition of SAP we can additionally assume there are no nontrivial relations between the sets $f_1(B_1)\smallsetminus f_1\circ g_1(A)$ and $f_2(B_2)\smallsetminus f_2\circ g_2(A)$.

\begin{thm}[\cite{bilge}]
  \label{thm:bm}
  Let $\mathcal K$ be a class of finite structures with the FAP, and let $M$ be the Fra\"iss\'e limit. Then the isomorphism relation on $\mathcal K_\omega$ is Borel reducible to the conjugacy relation on $\Aut(M)$.
\end{thm}

Many of the examples considered in \cite{conjugacy1} have the FAP, though some did not. The three structures considered in this article do not. Thus it is natural to seek a property that is weaker than the FAP but still strong enough to establish the conclusion of Theorem~\ref{thm:bm}. In the rest of this section, we undertake the beginnings of such a quest.

In the proof of Theorem~\ref{thm:bm} the authors do not use the FAP directly, but rather an extension construction: given some countable $A\in\mathcal K_\omega$, they build a structure $E(A)\in\mathcal K_\omega$ which witnesses all configurations which are in finitary in the following sense. (For later use we write the definition in slightly more general terms than in \cite{bilge}.)

\begin{defn}
  Let $\mathcal C$ be a class of countable structures and $A\in\mathcal C$. Given a finite subset $A_0\subset A$ and a quantifier-free type $\tau(x,\bar a)$ over $A_0$, we say that $\tau$ is an \emph{admissible finitary type} over $A$ (with respect to the class $\mathcal C$) if there exists $B\in\mathcal C$ such that $A$ is a substructure of $B$ and $B$ contains a witness for $\tau$.
\end{defn}

We remark that in order to build a structure $E(A)$ which contains witnesses for all admissible finitary types over $A$, it is sufficient for the ambient class $\mathcal K_\omega$ to have the SAP. However we have not been able to use the SAP alone to establish the conclusion of Theorem~\ref{thm:bm}. To establish the desired Borel reduction, one must be able to construct $E(A)$ in an explicit fashion. The following definition captures the precise definable and combinatorial hypotheses we need to carry out a proof.

\begin{defn}
  \label{defn:bap}
  Let $\mathcal C$ be a class of countable structures with the SAP. We say that $\mathcal C$ has the \emph{$n$-ary Borel amalgamation property} ($\text{BAP}_n$) if there is a Borel assignment $E\colon\mathcal C\to\mathcal C$ satisfying the following properties:
  \begin{itemize}
  \item For all $A\in\mathcal C$, $E(A)$ consists of $A$ together with $n$ witnesses $x^{\tau}_1,\ldots,x^{\tau}_n$ for each admissible finite type $\tau$ over $A$;
  \item $E(A)$ has an automorphism which fixes $A$ and cycles the witnesses $x^\tau_1\mapsto\ldots\mapsto x^\tau_n\mapsto x^\tau_1$;
  \item For admissible finite types $\tau,\tau'$, the map $x^\tau_i\mapsto x^{\tau'}_i$ is an isomorphism of finite structures; and
  \item Given $A,A'\in\mathcal C$ and isomorphism $\alpha\colon A\to A'$, the natural extension $\bar\alpha\colon E(A)\to E(A')$ is in fact an isomorphism. (Here $\alpha$ naturally extends to the admissible finite types over $A$, and hence to $E(A)$ by $\bar\alpha(x^\tau_i)=x^{\alpha(\tau)}_i$.)
  \end{itemize}
\end{defn}

We observe that $\text{FAP}\implies\text{BAP}_n\implies\text{SAP}$. Indeed, the first implication is a consequence of the results of \cite{bilge}, and the second implication holds by definition. Moreover the first implication is not reversible, since the class of countable partial orders satisfies $\text{BAP}_2$ but not FAP. (Here one can establish $\text{BAP}_2$ using Lemma~\ref{lem:po-amalg}.) We conjecture that the second implication is not reversible as well. For example the class of countable tournaments satisfies the SAP but does not apparently satisfy any $\text{BAP}_n$, an issue we will return to later on.

We are finally ready to state our generalization of Theorem~\ref{thm:bm}.

\begin{thm}
  \label{thm:bap}
  Let $\mathcal K$ be a class of finite structures such that $\mathcal K_\omega$ has the $\text{BAP}_n$ for some $n\geq2$, and let $M$ be the Fra\"iss\'e limit. Then the isomorphism relation on $\mathcal K_\omega$ is Borel reducible to the conjugacy relation on $\Aut(M)$.
\end{thm}

The details of the proof are quite similar to those of Theorem~\ref{thm:bm}. To illustrate these details, we show now how it works in the special case of partial orders.

\begin{proof}[Sketch of alternate proof of Theorem \ref{thm:partial}]
  Given a countable partial order $P$, we build a copy $Q_P$ of $\mathcal P$ and an automorphism $\phi_P$ of $Q_P$ in such a way that $P\cong P'$ if and only if $\phi_P$ is conjugate to $\phi_{P'}$. This will be done in stages $Q_P^n$ and $\phi_P^n$ as before.

  Let $Q_P^0$ be a copy of $P$, and let $\phi_P^0$ be the identity on $Q_P^0$. Next suppose $Q_P^n$ and $\phi_P^n$ have been constructed. Let $Q_P^{n+1}=E(Q_P^n)$, where the function $E$ is given by $\text{BAP}_2$.  In particular, we set $x^{\tau}_1 \perp x^{\tau}_2$. Let $\overline{\phi_P^n}$ be the natural extension of $\phi_P^n$ to all of $Q_P^{n+1}=E(Q_P^n)$, and let $\psi$ be the automorphism of $Q_P^{n+1}$ which fixes $Q_P^n$ and exchanges all witness pairs. We then let $\phi_P^{n+1}=\psi\circ\overline{\phi_P^n}$.

  Letting $Q_P=\bigcup Q_P^n$ and $\phi_P=\bigcup \phi_n$, it is easy to see that $Q_P$ is a copy of $\mathcal{P}$ and $\phi_P$ is an automorphism of it. If $P\cong P'$, then the requirements on the extension map $E$ directly imply that $\phi_P$ is conjugate to $\phi_{P'}$. Conversely, since $P$ can be recovered as the set of fixed points of $\phi_P$, we clearly have that if $\phi_P$ is conjugate to $\phi_{P'}$ then $P\cong P'$.
\end{proof}

Although Theorem~\ref{thm:bap} is strong enough to subsume some of the results from \cite{conjugacy1} as well as the case of partial orders, we have also been able to establish its conclusion even in cases when $\text{BAP}_n$ is not present. For example, we have already remarked that the class of tournaments does not apparently satisfy $\text{BAP}_n$. It is possible to use triplet witnesses $\{x_1^\tau,x_2^\tau,x_3^\tau\}$ as in $\text{BAP}_3$, and set $x_1^\tau\to x_2^\tau\to x_3^\tau\to x_1^\tau$. However it is not clear how to define the edges between $x^{\tau},x^{\tau'}$ for $\tau\neq\tau'$. Instead, we restricted to the subclass of linear orders, and made all such choices according to an inherited linear order on the types.

We have a similar situation in this article with the semigeneric multipartite digraph. It is possible to use twin witnesses as $x_1^\tau, x_2^\tau$ as in $\text{BAP}_2$; the two witnesses must belong to the same antichain and we may set $x_1^\tau \perp x_2^\tau$. On the other hand, we could not assign the edges between $x_1^{\tau} , x_1^{\tau'}$ in a Borel fashion, and relied on the subclass of linear orders once again.

To conclude, we remark that while Definition~\ref{defn:bap} is somewhat specialized, some form of Borel amalgamation may be of broader interest. It is well-known that if $\mathcal K$ has the SAP, then $\mathcal K_\omega$ has the SAP too (and similarly for the FAP). This can be established by either iteratively applying the finitary amalgamation property, or using a compactness argument. In either proof, one does not arrive at an explicit and uniquely determined amalgam: it relies on the choice of the enumeration of the finite substructures, or on the weak form of the Axiom of Choice that is the compactness theorem. A Borel form of SAP such as $\text{BAP}_1$ remedies this use of AC by isolating classes where countable amalgamation can be done in an explicit way.

\bibliographystyle{halpha}
\begin{singlespace}
  \bibliography{summer,conjugacy}
\end{singlespace}

\end{document}